\documentclass[a4paper, 11pt,reqno]{amsart}

\usepackage{pinlabel}
\usepackage{svg} 
\usepackage[all,cmtip]{xy}
\usepackage{amsmath,amssymb,amsthm, amscd, xypic}
\usepackage{mathtools}
\usepackage[english]{babel}
\usepackage{tikz-cd, enumerate}

\usepackage[normalem]{ulem}

\usepackage{todonotes}

\newtheorem{introthm}{Theorem}

\usepackage{xcolor}

\colorlet{GREEN}{green}
\colorlet{BLUE}{blue}

\definecolor{darkgreen}{rgb}{0,0.50,0} 
\definecolor{darkred}{rgb}{0.55,0,0}
\definecolor{darkblue}{rgb}{0,0,0.6}

\usepackage[pdfborder=0,pagebackref,colorlinks,citecolor=darkgreen,linkcolor=darkred,urlcolor=darkblue]{hyperref}
\addto\extrasenglish{
    
}

\newcommand\redout{\bgroup\markoverwith
{\textcolor{red}{\rule[.5ex]{2pt}{0.4pt}}}\ULon}
\def\makeautorefname#1#2{\expandafter\def\csname#1autorefname\endcsname{#2}}
\makeautorefname{equation}{Equation}%
\makeautorefname{footnote}{footnote}%
\makeautorefname{item}{item}%
\makeautorefname{figure}{Figure}%
\makeautorefname{table}{Table}%
\makeautorefname{part}{Part}%
\makeautorefname{appendix}{Appendix}%
\makeautorefname{chapter}{Chapter}%
\makeautorefname{section}{Section}%
\makeautorefname{subsection}{Section}%
\makeautorefname{subsubsection}{Section}%
\makeautorefname{paragraph}{Paragraph}%
\makeautorefname{subparagraph}{Paragraph}%
\makeautorefname{theorem}{Theorem}%
\makeautorefname{thm}{Theorem}%
\makeautorefname{introthm}{Theorem}%
\makeautorefname{addm}{Addendum}%
\makeautorefname{mainthm}{Main theorem}%
\makeautorefname{corollary}{Corollary}%
\makeautorefname{cor}{Corollary}%
\makeautorefname{lemma}{Lemma}%
\makeautorefname{lem}{Lemma}%
\makeautorefname{sublemma}{Sublemma}%
\makeautorefname{sublem}{Sublemma}%
\makeautorefname{subl}{Sublemma}%
\makeautorefname{prop}{Proposition}%
\makeautorefname{property}{Property}
\makeautorefname{pro}{Property}
\makeautorefname{sch}{Scholium}%
\makeautorefname{step}{Step}%
\makeautorefname{conject}{Conjecture}%
\makeautorefname{conj}{Conjecture}%
\makeautorefname{questn}{Question}
\makeautorefname{quest}{Question}
\makeautorefname{qn}{Question}
\makeautorefname{definition}{Definition}%
\makeautorefname{defn}{Definition}%
\makeautorefname{defi}{Definition}%
\makeautorefname{def}{Definition}%
\makeautorefname{dfn}{Definition}%
\makeautorefname{notation}{Notation}
\makeautorefname{notn}{Notation}
\makeautorefname{rem}{Remark}%
\makeautorefname{rems}{Remarks}%
\makeautorefname{rmk}{Remark}%
\makeautorefname{rk}{Remark}%
\makeautorefname{remarks}{Remarks}%
\makeautorefname{rems}{Remarks}%
\makeautorefname{rmks}{Remarks}%
\makeautorefname{rks}{Remarks}%
\makeautorefname{example}{Example}%
\makeautorefname{examp}{Example}%
\makeautorefname{exmp}{Example}%
\makeautorefname{example}{Example}%
\makeautorefname{exa}{Example}%
\makeautorefname{axiom}{Axiom}%
\makeautorefname{axi}{Axiom}%
\makeautorefname{ax}{Axiom}%
\makeautorefname{case}{Case}%
\makeautorefname{claim}{Claim}%
\makeautorefname{clm}{Claim}%
\makeautorefname{assumpt}{Assumption}%
\makeautorefname{asses}{Assumptions}%
\makeautorefname{conclusion}{Conclusion}%
\makeautorefname{concl}{Conclusion}%
\makeautorefname{conc}{Conclusion}%
\makeautorefname{cond}{Condition}%
\makeautorefname{const}{Construction}%
\makeautorefname{con}{Construction}%
\makeautorefname{construction}{Construction}%
\makeautorefname{criterion}{Criterion}%
\makeautorefname{criter}{Criterion}%
\makeautorefname{crit}{Criterion}%
\makeautorefname{exercise}{Exercise}%
\makeautorefname{exer}{Exercise}%
\makeautorefname{exe}{Exercise}%
\makeautorefname{problem}{Problem}%
\makeautorefname{problm}{Problem}%
\makeautorefname{prob}{Problem}%
\makeautorefname{prob}{Problem}%
\makeautorefname{soln}{Solution}%
\makeautorefname{sol}{Solution}%
\makeautorefname{sum}{Summary}%
\makeautorefname{oper}{Operation}%
\makeautorefname{obs}{Observation}%
\makeautorefname{ob}{Observation}%
\makeautorefname{conv}{Convention}%
\makeautorefname{cvn}{Convention}%
\makeautorefname{warn}{Warning}%
\makeautorefname{note}{Note}%
\makeautorefname{fact}{Fact}%
\makeautorefname{ouch0}{Counterexample}%
\newtheorem{theorem}{Theorem}[section]
\newtheorem{cor}{Corollary}[section]
\newtheorem{prop}{Proposition}[section]

\newtheorem{lem}{Lemma}[section]
\newtheorem{lemma}{Lemma}[section]
\newtheorem*{conj}{Conjecture}

\theoremstyle{definition}
\newtheorem{definition}{Definition}[section]

\newtheorem*{conv}{Convention}

\newtheorem{example}{Example}[section]

\newtheorem{rem}{Remark}[section]

\newtheorem{construction}{Construction}[section]

\makeatletter
\let\c@cor=\c@thm
\let\c@prop=\c@thm
\let\c@proposition=\c@thm
\let\c@theorem=\c@thm
\let\c@lem=\c@thm
\let\c@definition=\c@thm
\let\c@conj=\c@thm
\let\c@defn=\c@thm
\let\c@df=\c@thm
\let\c@exmp=\c@thm
\let\c@example=\c@thm
\let\c@rem=\c@thm
\let\c@lemma=\c@thm
\let\c@sch=\c@thm
\let\c@con=\c@thm
\let\c@equation\c@thm
\let\c@construction\c@thm
\makeatother

\newcommand{\res}{\mathrm{res}}

\newcommand{\SK}{\mathrm{SK}}
\newcommand{\SKg}{SK^{G}}
\newcommand{\SKH}{SK^{H}}
\newcommand{\SKk}{SK^{K}}

\newcommand{\id}{\mathrm{id}}
\newcommand{\grp}{\mathrm{grp}}
\newcommand{\Burn}{\mathrm{Burn}}

\def\C{\mathcal{C}}
\def\d{\partial}

\newcommand{\R}{\mathbb R}
\newcommand{\Z}{\mathbb Z}

\newcommand{\cat}[1]{\textup{\textbf{{#1}}}}

\newcommand{\Ab}{\cat{Ab}}
\newcommand{\aug}{\textup{aug}}

\newcommand{\Mor}{\textup{Mor}}

\newcommand{\Mnfld}{\textup{Mfd}}

\newcommand{\tr}{\textup{tr}\,}

\newcommand{\RP}{\mathbb{RP}}

\newcommand{\ChPerf}{\textup{Ch}^{\textup{hb}}}

\newcommand{\Span}{\textup{Span}}

\newcommand\GSK{\mathop{\mbox{$G$-$SK$}}}
\newcommand\CpSK{\mathop{\mbox{$C_p$-$SK$}}}

\newcommand{\calC}{\mathcal{C}}
\newcommand{\Hc}{\mathcal{H}}
\newcommand{\Vc}{\mathcal{V}}

\newcommand{\SKG}{SK^{G,\partial}_n}
\newcommand{\Mfld}{\mathrm{Mfld_n^{G,\partial}}}
\newcommand{\MfldH}{\mathrm{Mfld_n^{H,\partial}}}
\newcommand{\ob}{\mathrm{ob}}

\newcommand{\Kso}[1]{K_0^\square(#1)}

\newcommand{\dsquare}[4]{
	\begin{tikzcd}[ampersand replacement=\&]
	#1\ar[r, >->]\ar[d, hook] \ar[dr, phantom, "\square"]\& #2 \ar[d, hook]\\
	#3 \ar[r, >->] \& #4
	\end{tikzcd}}

\newcommand{\dsquaref}[8]{
	\begin{tikzcd}[ampersand replacement=\&]
	#1\ar[r, >->,"#5"]\ar[d, swap,hook,"#6"] \ar[dr, phantom, "\square"]\& #2 \ar[d, hook,"#8"]\\
	#3 \ar[r, >->,"#7"] \& #4
	\end{tikzcd}}

\newcommand{\rtail}{\rightarrowtail}
\newcommand{\htail}{\hookrightarrow}

\newlength{\storeparskip}
\author[M. Merling]{Mona Merling}
\address{Department of Mathematics, University of Pennsylvania}
\email{mmerling@upenn.edu}
\author[M. Ng]{Ming Ng}
\address{Graduate School of Informatics, Nagoya University}
\email{ng.ming.k0@a.mail.nagoya-u.ac.jp}
\author[J. Semikina]{Julia Semikina}
\address{Paul Painlevé Mathematics Laboratory, University of Lille}
\email{iuliia.semikina@univ-lille.fr}
\author[A. Send\'{o}n Blanco]{Alba Send\'{o}n Blanco}
\address{Department of Mathematics, Vrije Universiteit Amsterdam}
\email{a.sendon.blanco@vu.nl }
\author[L. Williams]{Lucas Williams}
\address{Department of Mathematics, Binghamton University}
\email{lwilli39@binghamton.edu}

\date{}

\makeatletter
\@namedef{subjclassname@2020}{\textup{2020} Mathematics Subject Classification}
\makeatother

\title{Scissors congruence $K$-theory for equivariant manifolds}

\begin{document}

\begin{abstract} 
We introduce a scissors congruence $K$-theory spectrum which lifts the equivariant scissors congruence groups for compact $G$-manifolds with boundary, and we show that on $\pi_0$ this is the source of a spectrum level lift of the  Burnside ring valued equivariant Euler characteristic of a compact $G$-manifold. We also show that the equivariant scissors congruence groups for varying subgroups assemble into a Mackey functor, which is a shadow of a conjectural higher genuine  equivariant structure.
\end{abstract}

\maketitle

\vspace{-2ex}

\begingroup%
\setlength{\parskip}{\storeparskip}
\setcounter{tocdepth}{1}
\tableofcontents
\endgroup%

\setcounter{section}{0}

\section{Introduction}
 Hilbert's Third Problem asked about scissors congruence invariants of polyhedra in 3 dimensions, which were completely classified over the following 60 years in \cite{dehn, Sydler}.  The problem of determining all scissors congruence invariants of polyhedra in dimensions higher than 4 is still open. For many years, studying scissors congruence revolved around calculating the abelian group of polytopes in a given geometry up to cut-and-paste. In \cite{inna_scissorsKth}, Zakharevich recasts scissors congruence in terms of $K$-theory. This involves constructing a $K$-theory spectrum that not only recovers the classical group of polytopes on $\pi_0$, but also encodes deeper information about scissors congruences in its higher $K$-groups. The study of this new ``higher scissors congruence"  of polyhedra has flourished in recent years \cite{bgmmz, scissors_thom, klmms}.

In the 70s, Karras, Kreck, Neumann, and Ossa \cite{SKbook} introduced a definition of scissors congruence for closed manifolds, also called the $SK$-relation (“schneiden und kleben”, German for “cut and paste”). Given a closed smooth 
manifold, $M$, one can cut it along a codimension 1 separating submanifold, $\Sigma$, with trivial normal bundle and paste back the two pieces along a  diffeomorphism $\Sigma \to \Sigma$  to obtain a new manifold, $M'$.  Two manifolds, $M$ and $M'$, are {\em $SK$-equivalent} if one can be obtained from the other via a finite sequence of $SK$-operations. The groups $\SK_n$, of  diffeomorphism classes of $n$-manifolds up to the $SK$-relation, are computed in \cite{SKbook}. Unlike the case of polyhedra, the Euler characteristic completely determines the scissors congruence class of an unoriented closed manifold in any dimension. 

The notions of $SK$-equivalence and $SK$-groups were generalized to the setting of manifolds with boundary in \cite{witpaper}, which allowed the authors to leverage the framework of  $K$-theory for squares categories \cite{CKMZsquarescategories} to construct a $K$-theory spectrum recovering the $SK$-group as  $\pi_0$. As an application of this result, the authors construct a map of spectra from this $K$-theory spectrum to the $K$-theory of the integers which recovers the Euler characteristic.

In the present paper, we extend the results of \cite{witpaper} to $G$-manifolds and equivariant $SK$-groups.  These groups were extensively studied by many authors such as Rowlett \cite{rowlett1971additive}, J\"anich \cite{janich1969invariants}, Kosniowski \cite{kosniowski}, Hara and Koshikawa \cite{hara1997cutting}, and Komiya \cite{komiya2003cutting},  and have proved to be significantly more difficult to analyze than their non-equivariant counterparts. While the authors of \cite{SKbook} completely classified $SK$-invariants, a complete description of $\GSK$-invariants  is still unknown even in the case of finite abelian groups. 

There is a classical notion of equivariant Euler characteristic, defined for $G$-CW complexes as an alternating sum of cells, valued in the Burnside ring. This has been studied in  \cite{tomDieckbook, luck2005burnside, LuckRosenberg} and is closely related to the Euler characteristics of fixed points, yet does not appear in previous treatments of equivariant scissors congruence for manifolds. We note that this is an example of $\GSK$-invariant. However, in contrast to the non-equivariant case, the equivariant Euler characteristic (or the Euler characteristics of fixed points) does not fully determine the $\GSK$-equivalence class of unoriented $G$-manifolds. Even though the study of $G$-spaces often reduces to the study of their $H$-fixed points for all $H\leq G$, equivariant scissors congruence of manifolds is more subtle. There is a finer set of invariants, the slice type Euler characteristics, which Kosniowski \cite{kosniowski} shows are a full set of invariants when $G$ is a finite abelian group of odd order. He further conjectures that the slice type Euler characteristics are a complete set of invariants for any finite group.

 A natural question, which has been overlooked until now, is whether the equivariant $\SK^H$-groups, for varying subgroups $H$ of a group $G$, form a Mackey functor. Mackey functors are the analogues of abelian groups in equivariant homotopy theory. 
In particular, they arise as homotopy groups of genuine $G$-spectra. Our first main result establishes the existence of this extra structure on equivariant $SK$-groups, which is a shadow of conjectural higher genuine equivariant structure. 

\begin{introthm}\label{thmA}
Let $G$ be a finite group. The equivariant $\SKH$-groups, varying over subgroups $H \leq G$, form a Mackey functor.
\end{introthm}

We use the framework of squares $K$-theory from  \cite{CKMZsquarescategories} to  lift the equivariant $SK$-groups to the spectrum level, generalizing the result in \cite{witpaper}. Given a finite group $G$, we construct a scissors congruence $K$-theory spectrum, denoted by $K^\square(\Mfld)$, which recovers cut-and-paste groups of $G$-manifolds with boundary as $\pi_0$. For each subgroup $H \leq G$ there exists a map of spectra from $K^\square(\MfldH)$ to the $H$-fixed points of the equivariant $A$-theory of \cite{MMEquivariantATheory}. To lift the equivariant Euler characteristic to $K$-theory spectra, we use the equivariant linearization map of \cite{CCMlinearization}, the source and target of which are equivariant $A$-theory and the equivariant $K$-theory of the constant coefficient system $\underline{\Z}$, respectively.

\begin{introthm}\label{thmB} For each subgroup $H \leq G$, there exists a map of spectra 
$$K^\square(\MfldH) \to K_G(\underline{\mathbb{Z}})^H,$$  which recovers the equivariant Euler characteristic, valued in the Burnside ring, on $\pi_0$.
\end{introthm}

It is straightforward to check that on $\pi_0$, the map $K^\square(\MfldH) \to K_G(\underline{\mathbb{Z}})^H$  is a map of Mackey functors. This suggests the following conjecture.\footnote{The conjectured $G$-spectrum and lift of the equivariant Euler characteristic map were constructed in the recently announced preprint \cite{CalleChan}.}

\begin{conj}
The $\SK^G$ Mackey functor is $\pi_0$ of a genuine $G$-spectrum  $K^\square_G(\Mfld)$, and there is a map of genuine $G$-spectra $K^\square_G(\Mfld)\to K_G(\underline{\mathbb{Z}})$, which lifts the equivariant Euler characteristic.  \end{conj}

\subsection*{Overview} The paper is organized as follows. In \autoref{section: GSK inv}, we introduce equivariant $SK$-groups for manifolds with boundary and contextualize them with known results on the $\GSK$-invariants of closed manifolds. We also discuss the Burnside ring valued equivariant Euler characteristic. In \autoref{section: Mackey}, we show how the $\SKH$ groups, for $H \leq G$, assemble into a Mackey functor. In \autoref{section :equivariant squares spectrum}, we construct an equivariant scissors congruence $K$-theory spectrum, for each group $H$, recovering $\SKH$ on $\pi_0$. This spectrum is then related to $A_G(*)^H$, the $H$-fixed points of equivariant $A$-theory of a point. Via equivariant linearization, it is further related to $K_G(\underline{\mathbb{Z}})^H$, the $H$-fixed points of the equivariant $K$-theory of the coefficient system $\underline{\Z}$. The key result of this section is that the resulting map to $K_G(\underline{\mathbb{Z}})^H$ lifts the equivariant Euler characteristic to the level of $K$-theory spectra.

\begin{conv} In this paper, $G$ will always be a finite group unless stated otherwise. We use the term {\em $G$-manifold} to
mean an unoriented compact smooth manifold (possibly with boundary) equipped with a smooth $G$-action.     
\end{conv}

\subsection*{Acknowledgments} 
 We thank the organizers of the Collaborative Research Workshop on K-theory and Scissors Congruence, the Vanderbilt University's Mathematics Department for their hospitality during the workshop, and the NSF for supporting this workshop as part of a Focused Research Collaboration grant. The authors would like to acknowledge contributions  to this paper arising from conversations with Maxine Calle, David Chan, Johannes Ebert, Tom Goodwillie, Renee S. Hoekzema, Achim Krause, Wolfgang L{\"u}ck, Cary Malkiewich, and Antoine Touz\'e.  M.M. was partially supported by NSF DMS grants CAREER 1943925 and FRG 2052988. M.N. was partially supported by EPSRC Grant EP/V028812/1 and a FY2024 JSPS Postdoctoral Fellowship (Short-Term). J.S. was partially supported by the Labex CEMPI (ANR-11-LABX-0007-01). L.W. was partially supported by NSF DMS-2052923.

\section{Equivariant cut-and-paste invariants} \label{section: GSK inv}

In this section, we introduce equivariant cut-and-paste groups for $G$-manifolds with boundary, simultaneously generalizing both the definition from \cite{witpaper} to the equivariant case, and  the equivariant definition from \cite{SKbook} to $G$-manifolds with boundary. 
\subsection{$SK$-equivalence for $G$-manifolds} The $SK$-groups (German ``schneiden und kleben"=``cut and paste") of closed $G$-manifolds were introduced in \cite{SKbook}. In order to define a scissors congruence $K$-theory for $G$-manifolds, we must work in the category of $G$-manifolds with boundary so that the category contains the pieces in a cut-and-paste operation. Our definition is different from the definition in \cite{hara1997cutting}, where cutting  the boundary is allowed---as in the non-equivariant definition from \cite{witpaper}, we define the cut-and-paste operation away from the existing boundary.

Let $M$ be a $G$-manifold. Let $\Sigma\subseteq M$ be a $G$-invariant codimension 1 smooth submanifold, disjoint from $\partial M$, with normal bundle given by $\Sigma\times \R$ such that $G$ acts trivially on $\R$. As in the non-equivariant case, there is no loss in generality in requiring that $\Sigma$ separates $M$ into two disjoint manifolds. Define a $\GSK$-operation on $M$ as follows.

\begin{definition}  \label{def: GSK operation}
Cut $M$ along $\Sigma$ as above, obtaining the disjoint union of two $G$-manifolds $M_1$ and $M_2$, each with part of their boundary equivariantly diffeomorphic to $\Sigma$. Then paste back the two pieces together along an equivariant diffeomorphism  $\phi \colon \Sigma \to \Sigma$. We say that $M_1\cup_\phi M_2$ is obtained from $M$ by a \textit{$\GSK$-operation}.

\end{definition}
We emphasize that we do not allow boundaries to be cut, and we require that all boundaries which come from cutting are pasted back together, leaving the original boundary of a manifold untouched by the cut-and-paste operation. 

   \begin{definition}
We say that two $G$-manifolds are $\GSK$-equivalent if one can be obtained from the other by a finite sequence of $\GSK$-operations. 
   \end{definition}

  \autoref{SKexample} depicts an example of two distinct $C_2$-actions on $S^2\sqcup T^2$ which are $C_2$-$SK$-equivalent.  The red lines in the pictures indicate the $C_2$-fixed points. The diagonal action on the torus is given by reflection across the diagonal of the square before identifying opposite sides.\footnote{In \cite[Theorem 1.11]{dugger2019involutions}, Dugger classified the six possible $C_2$-actions on the torus. The one we call the diagonal action is isomorphic to $T^{\mathrm{anti}}_0$+$[S^{1,0}-\text{antitube}]$ in his notation. It is given by cutting out two disjoint disks from $S^2$ with an antipodal action and sewing in a cylinder with a flip action (i.e. an $S^{1,0}$-antitube) and is depicted in the middle section of \autoref{SKexample}.} 
    \begin{figure} 
\centering
\includegraphics[width=\textwidth]{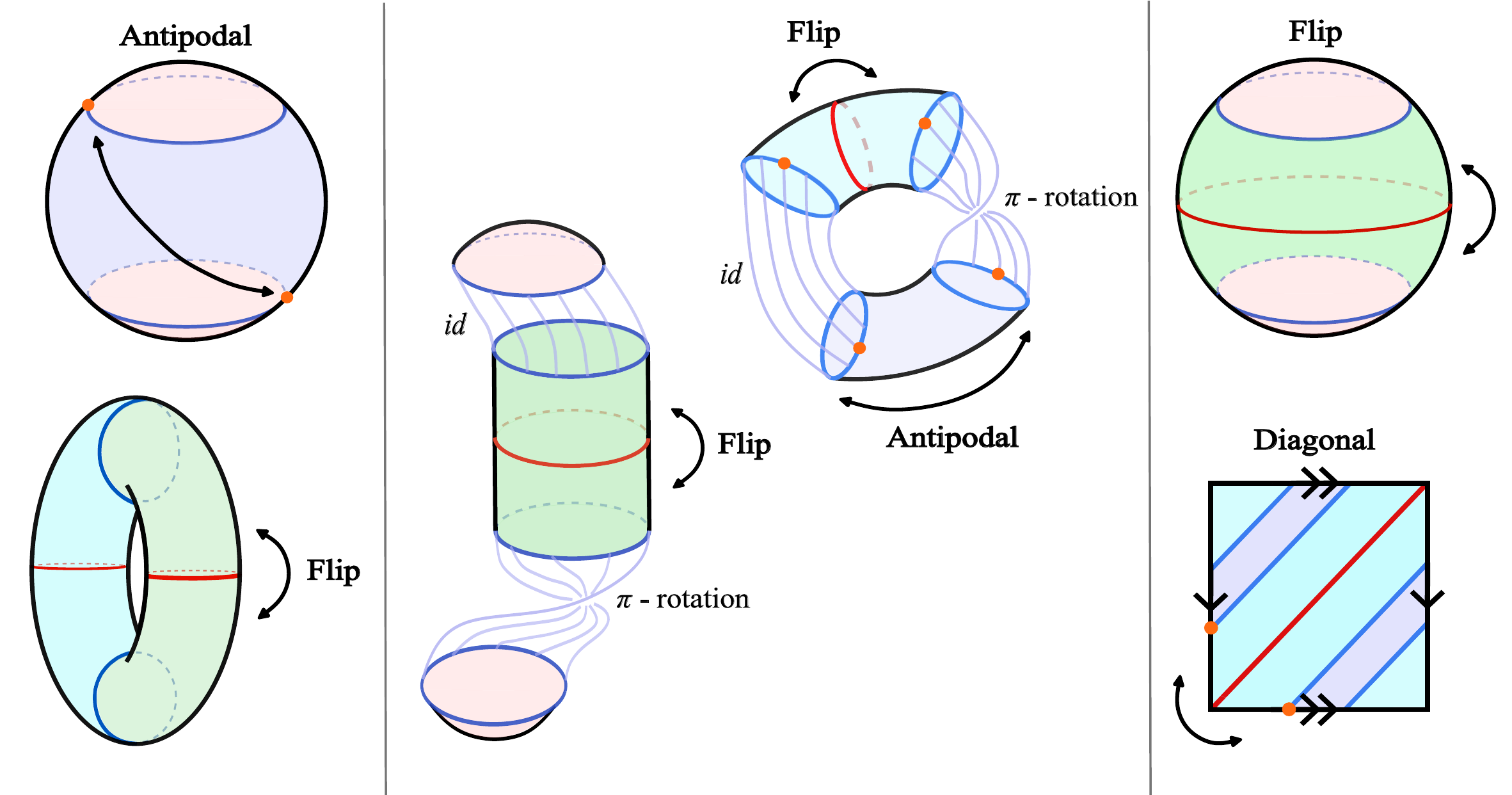}
\caption{$C_2$-SK equivalence between $S^2_{\mathrm{antipodal}}\sqcup T^2_{\mathrm{flip}}$ and $S^2_{\mathrm{flip}}\sqcup T^2_{\mathrm{diagonal}}$} 
\label{SKexample}
\end{figure}

\begin{definition}
    Let $\mathcal{M}_{n}^{G, \d}$ be the monoid of $G$-diffeomorphism classes of $n$-dimensional $G$-manifolds with boundary under disjoint union. A $\GSK$-invariant is an abelian group valued map of monoids out of $\mathcal{M}_{n}^{G, \d}$, which is constant on $\GSK$-equivalence classes. 
\end{definition}
We define a group with the universal property that every $\GSK$-invariant factors through it. 
\begin{definition}
The \emph{equivariant scissors congruence group $\SKG$ of $G$-manifolds} is the quotient of the group completion $(\mathcal{M}_{n}^{G, \d})^{\grp}$ by the $\GSK$-equivalence relation.
\end{definition}

\begin{rem}\label{SKsum}
In \cite{witpaper}, the authors show that non-equivariantly the scissors congruence group of manifolds with boundary splits into a direct sum of the scissors congruence group of closed manifolds and the group completion of the monoid of nullcobordant manifolds. The latter is the group of all possible boundaries. It is unknown whether a similar decomposition exists for equivariant scissors congruence groups. The proof of the non-equivariant splitting relies on the classification of $SK$-invariants. Since $\GSK$-invariants are not entirely classified, the proof of the splitting in \cite{witpaper} does not readily generalize to the equivariant setting.
\end{rem}


\subsection{Equivariant Euler characteristic}\label{equivEulersection}  In this subsection, we define the equivariant Euler characteristic. The equivariant Euler characteristic of a $G$-manifold is a $\GSK$-invariant. While the Euler characteristic of a manifold completely determines its $SK$ equivalence class, we will see that this is not true equivariantly.

\begin{definition}\label{Burnsidering}
 The \emph{Burnside ring} $\Burn(G)$ is the group completion of the monoid of isomorphism classes of finite $G$-sets under disjoint union. The commutative ring structure is given by taking products of $G$-sets.
\end{definition}

A $G$-CW complex  is built out of cells of type $G/H$, that is, of the form $G/H\times D^k$ for $H \leq G$, so that each fixed point subspace is a subcomplex. By \cite{illman},  every compact smooth $G$-manifold is homeomorphic to a finite $G$-CW complex.

\begin{definition}\label{Gchi}
    Let $X$ be a finite $G$-CW complex. The \emph{equivariant Euler characteristic}, $\chi_G(X)$, is defined as $$\chi_G(X)=\sum_k (-1)^k [\mathrm{Cell}_k(X)]\in \Burn(G),$$ where $\mathrm{Cell}_k(X)$ is the $G$-set of $k$-dimensional cells of $X$. 
\end{definition}


Let $C_G$ be the set of conjugacy classes of subgroups of $G$. Following \cite[\S 1]{luck2005burnside}, we note that the character map $\mathrm{char}^G\colon \Burn(G) \to \prod_{(H)\in C_G} \Z$ which, in the $(H)$-component, sends a $G$-set to the number of $H$-fixed points, is injective. For a finite $G$-CW complex $X$, the $(H)$-component of the image of $\chi_G(X)$ under this map is $\chi(X^H)$. Therefore, the equivariant Euler characteristic is invariant under $G$-weak equivalences and thus, independent of choosing a $G$-CW structure. 

The equivariant Euler characteristic of a $G$-CW complex $X$ can be expressed as (see e.g. \cite[Lemma 1.7]{luck2005burnside})
\begin{equation}\label{Eulerformula} 
    \chi_G(X)= \sum_{(H) \in C_G} \chi(X^H/WH,\ \bigcup_{K \gneq H}X^K /WH)[G/H],
\end{equation}
where $WH=N_G(H)/H$ is the Weyl group of $H$ in $G$. In particular, when $X$ has a free $G$-action, $\chi_G(X)=\chi(X/G)[G/e]$. 

The usual additivity formula for Euler characteristic still holds equivariantly,
$$\chi_G(X\cup Y)= \chi_G(X)+\chi_G(Y)-\chi_G(X\cap Y),$$
so the equivariant Euler characteristic $\chi_G$ of a $G$-manifold is a $\GSK$-invariant. Thus, the tuple of fixed point Euler characteristics $(\chi(X^H))_{(H)}\in \prod_{(H)} \Z$ is also a $\GSK$-invariant. Alternatively, we could directly observe that the Euler characteristics of the fixed points are $\GSK$-invariants as taking fixed points commutes with pushouts along closed inclusions.

The following example uses the equivariant Euler characteristic to demonstrate that $\GSK$-equivalence is a strictly finer relation than non-equivariant $SK$-equivalence for $G$-manifolds.

  \begin{example}
         Let $S^1_{\mathrm{triv}}$ be the circle with trivial $C_2$-action and let $S^1_{\mathrm{flip}}$ be the circle with $C_2$ acting by a reflection across a diameter, so that the flip action has exactly two fixed points. The Euler characteristics of the $C_2$-fixed points of these two spaces differ: $\chi((S^1_{\mathrm{triv}})^{C_2})=0$, while $\chi((S^1_{\mathrm{flip}})^{C_2})=2$, so the circle with trivial action is not $C_2$-$SK$ equivalent to the circle with flip action. 
  \end{example} 

  \subsection{Slice type Euler characteristic}

As noted in \autoref{SKsum}, if $G$ is trivial, then the Euler characteristic and the diffeomorphism class of the boundary determine the $SK$-class of an unoriented manifold (see \cite{SKbook} and \cite{witpaper}). In general, the equivariant Euler characteristic is not enough to determine the $\GSK$-class. This is illustrated by \autoref{Ex: euler char of fixed points is not enough}.

Equivariant $SK$-groups for closed manifolds were extensively studied in \cite{kosniowski} by using the notion of slice types (for details, see e.g. \cite[\S II.4]{janich1966differenzierbare}). This definition naturally extends to $G$-manifolds with boundary.

\begin{definition} 
Let $M$ be a $G$-manifold, let $x\in M$ and let $G_x$ be the stabilizer group of $x$.

	\begin{enumerate}
		\item The {\em slice type of $M$ at $x$} is the conjugacy class of the pair $[G_x,V_x]$,
		where $V_x$ is the non-trivial part of the $G_x$-representation $T_xM$, namely,
		$$T_xM=V_x\oplus \R^p \text{~and~} V_x^{G_x}=\{0\}.$$
       Note that if $x \in \partial M$ we consider $T_x(\partial M)$ instead since the normal direction to the boundary is trivial.

         \item A slice type of $G$ is a conjugacy class of a pair $[H,V]$, where $H$ is a subgroup of $G$ and $V$ is an isomorphism class of a real $H$-representation such that $V^H=0$. Here $V$ can be zero dimensional.
         
		\item Let $[H,V]$ be a slice type of $G$. The {\em $[H,V]$-stratum of $M$} is 
		$$M_{[H,V]}:=\{x\in M\mid [G_x,V_x]=[H,V]\}.$$
	\end{enumerate}
\end{definition}

The slice theorem (see e.g. \cite[Corollary VI.2.4]{bredon1972introduction}) guarantees that the slice type completely determines the local behavior of $M$. More precisely, it states that there exists a $G$-invariant open neighbourhood of $x\in M$ which is $G$-diffeomorphic to $G\times_{G_x} T_x M$. In particular, this ensures that $M_{[H,V]}$ is a smooth $G$-submanifold of $M$ for any slice type $[H,V].$ The submanifold $M_{[H,V]}$ might have boundary (which must be a submanifold of $\partial M$).

\begin{example}
    Let $D^2$ be the disk with $C_2$-action given by reflection across a diameter. The space of $C_2$-fixed points is the closed segment in the middle of the disk. For every point $x$ in the interior of the $C_2$-fixed points, $T_xD^2 = \R\oplus\R^\sigma$, where $\R^\sigma$ is the sign representation. Thus, $V_x=\R^\sigma$ in this case. For a  boundary point $x$ of the  $C^2$-fixed points, which lies on the boundary  of the disk, we  again have $V_x=\R^\sigma$. Thus, the stratum $D^2_{[C_2, \R^\sigma]}$ is the closed segment in the middle of the disk, a submanifold with boundary.
\end{example}

\begin{definition}
Let $M$ be a $G$-manifold and let $[H,V]$ be a slice  type of $G$. The corresponding \emph{slice type Euler characteristic} is
$$ \chi_{[H,V]}(M):=\chi(M_{[H,V]}).$$	
\end{definition}
For closed $G$-manifolds, the following lemma dates back to \cite{SKbook}.
\begin{lemma}
   The slice type Euler characteristic, $\chi_{[H,V]}$, is a $\GSK$-invariant for any slice type $[H,V]$ of $G$.
\end{lemma}
\begin{proof} Let $M$ be a $G$-manifold and $\Sigma$ be a $G$-invariant submanifold  of codimension 1 along which we are allowed to cut.
We claim that $\Sigma$ intersects $M_{[H,V]}$ transversally. By definition, $\Sigma$ and $M_{[H,V]}$ intersect transversally if
\[
   T_xM = T_x\Sigma \oplus T_xM_{[H,V]}\qquad \text{for all}\,x\in \Sigma \cap M_{[H,V]}.
\]

By the equivariant tubular neighbourhood theorem (see e.g. \cite[Theorem VI.2.2]{bredon1972introduction}), there exists a neighborhood of $\Sigma$ that is equivariantly diffeomorphic to $\Sigma\times (-\epsilon,\epsilon)$ for some $\epsilon>0$. Since $G$ acts trivially on $(-\epsilon,\epsilon)$, every point in $\{x\}\times (-\epsilon,\epsilon)$ has the same stabilizer group and slice representation as the original point $x$. Hence $\{x\} \times (-\epsilon,\epsilon)\subset M_{[H,V]}$. Therefore, 
\[
T_xM\cong T_x\Sigma\oplus T_x(-\epsilon,\epsilon) \subseteq T_x\Sigma\oplus T_xM_{[H,V]}.
\]
Since $T_x\Sigma\oplus T_xM_{[H,V]}$ is a subspace of $T_xM$, we conclude that $\Sigma$ and $M_{[H,V]}$ are transverse. 

The transversality guarantees that the intersection of the cut, $\Sigma$, with the stratum $M_{[H,V]}$ is a codimension 1 submanifold of $M_{[H,V]}$, which is an allowed equivariant cut of $M_{[H,V]}$. Hence, a $\GSK$-operation on $M$ results in a $\GSK$-operation on $M_{[H,V]}$. Therefore, the slice type Euler characteristic is a $\GSK$-invariant. We note that $M_{[H,V]}$ can be open, but its Euler characteristic is still preserved under cutting and pasting.
\end{proof}

By \cite[Theorem 5.2.1]{kosniowski}, if $G$ is an abelian group of odd order, then the slice type Euler characteristics completely determine the $\GSK$-class of a closed $G$-manifold. Next we provide an example showing that the equivariant Euler characteristic is not enough to distinguish $\GSK$ classes.

\begin{example} \label{Ex: euler char of fixed points is not enough}
    Let $C_p$ be a cyclic group of prime order with $p \geq 5$. It has $\frac{p-1}{2}$ two-dimensional irreducible representations over $\R$. Choose two different 2-dimensional irreducible real representations $V$ and $W$. Consider the following $C_p$-manifolds    
     \[
       M_1:=\RP(\R\times V)\  \text{and}\ 
       M_2:=\RP(\R\times W),
    \]
where $C_p$ acts trivially on $\R$. 

Let us now compute the Euler characteristic of their fixed points. Since $M_1$ and $M_2$ are both diffeomorphic to $\RP^2$, 
	$$\chi(M_0^{\{e\}})=\chi(M_1^{\{e\}}).$$ 
	 Since $V$ is a non-trivial irreducible representation of $C_p$, the only $C_p$-fixed point of $M_1$ is the class $[1:0]$ where $1\in \R$ and $0\in V$. The same is true for $M_2$, and so 
		$$\chi(M_0^{C_p})=1=\chi(M_1^{C_p}).$$ 
Hence, the Euler characteristics of the fixed points of $M_1$ and $M_2$ agree, and thus, their equivariant Euler characteristics agree in the Burnside ring
$$\chi_{C_p}(M_1)= \chi_{C_p}(M_2).$$
The slice type Euler chacteristics, and therefore the $\CpSK$ classes, of $M_1$ and $M_2$ differ. To check this claim, we use the following formula from \cite[p. 199]{kosniowski}. Given any subgroup $K \leq G$ (where $G$ is finite abelian of odd order), and any non-trivial irreducible $K$-representation, $T$, it holds that
$$\chi_{[H, U]}(G\times_K \RP(\R\times T))=\begin{cases} |G/H|,\qquad \text{if} \, H=K, ~U=T;\\
0, \qquad\qquad\, \text{otherwise.}
\end{cases}$$
Since $G\times_G N\cong N$ for any  $G$-manifold $N$ we have
\[
\chi_{[C_p,V]}(M_1)=|C_p/C_p|=1,
\]
\[
\chi_{[C_p,V]}(M_2)=0.
\]
\end{example}

\section{$\SKg$ is a Mackey functor} \label{section: Mackey}

In this section, we show that the equivariant $SK$-groups assemble into a Mackey functor.

\subsection{Mackey functors}    We begin by stating the definition of Mackey functors as certain functors out of the Burnside category. This definition originates with \cite{lindner1976remark}.

   \begin{definition} 
Let $\mathcal{C}$ be a category. Its \textit{span category, $\Span(\mathcal{C})$,} has the same objects as $\mathcal{C}$, and the morphisms, $\Mor(X, Y)$, are equivalence classes of span diagrams $X\leftarrow U \to Y$. Two diagrams $X\leftarrow U \to Y$ and $X\leftarrow U' \to Y$ are equivalent if there is an isomorphism $U \cong U'$ that makes both triangles commute. Composition is given by the pullback of spans.
   \end{definition}

Let $\mathcal{F}_G$ be the category of finite $G$-sets. The hom-sets of $\Span(\mathcal{F}_G)$ are commutative monoids where the sum of $X\leftarrow U \to Y$ and $X\leftarrow V \to Y$ is given by $X\leftarrow U \sqcup V \to Y$, and the additive identity is represented by $X \leftarrow \varnothing \to Y$.

   \begin{definition}
    The \textit{Burnside category, $\mathcal{B}_G$}, is an additive category obtained from $\Span(\mathcal{F}_G)$ by applying group completion to every hom-set.
   \end{definition}

    \begin{definition} \label{def Mackey in terms of spans}
    A \textit{$G$-Mackey functor}  is an additive functor from the Burnside category to the category of abelian groups.
    \end{definition}

   An equivalent axiomatic definition of a $G$-Mackey functor  $\underline{M} \colon \mathcal{B}_G \to \Ab$ consists of the data of an abelian group $\underline{M}(G/H)$ for every subgroup $H$ of $G$ and restriction, transfer, and conjugation homomorphisms
\[
\res^H_K\colon \underline{M}(G/H) \to \underline{M}(G/K),\ \ 
\tr^H_K\colon \underline{M}(G/K) \to \underline{M}(G/H),\ \ 
c_g\colon \underline{M}(G/H) \to \underline{M}(G/gHg^{-1})
\]


\noindent    for all subgroups $K \leq H \leq G$ and elements $g \in G$. These maps are subject to certain compatibility conditions, notably the so-called double coset formula. We refer the reader to \cite{webb2000guide} for the details of the compatibilities in the axiomatic definition, and for the comparison of the two definitions.

The collection of $\SKH$-groups, with $H \leq G$, naturally supports restriction, transfer, and conjugation operations. However, directly checking that they satisfy the axioms in the definition in \cite{webb2000guide} would require a laborious verification of all the compatibility conditions. To avoid this, we will leverage both definitions and their equivalence. We will verify \autoref{def Mackey in terms of spans} and then check that restriction, transfer, and conjugation coincide with the expected formulas. 

\subsection{Singular $SK$-groups} For the purpose of endowing equivariant $SK$-groups with a Mackey functor structure, it will be convenient to adopt a more general notion of equivariant $SK$-groups relative to a $G$-space $X$. We briefly recall the definition below, following \cite[Chapter 1]{SKbook}. 

  \begin{definition} 
  
   Let $X$ be a $G$-space. A \textit{singular $n$-manifold in $X$} is an equivalence class of a pair $(M, f)$, where $M$ is a compact $n$-dimensional $G$-manifold  and $f \colon M \to X$ is a continuous $G$-map; with $(M, f) \sim (M', f')$ if there exists a $G$-diffeomorphism $\varphi\colon M\to M'$ that makes the following triangle commute:
\[
         \begin{tikzcd}
            M \arrow{rr}{\varphi} \arrow{dr}[swap]{f} & & M' \arrow{dl}{f'} \\
            ~ & X. &
         \end{tikzcd}
\]
   Let $\mathcal{M}^{G}(X)$ denote the monoid of singular $n$-manifolds in $X$ under  disjoint union.  We write $\mathcal{M}^{G}$ when $X$ is a point. Throughout this section, the dimension $n$ is fixed and is omitted from the notation. Note that manifolds may have boundaries, but we suppress the boundary symbol $\partial$ for better readability. For instance, $\SKg$ here has the same meaning as $\SKG$ in the preceding section.
  \end{definition}

  \begin{definition} 
  
  We say that the singular $n$-dimensional $G$-manifolds $(M,f)$ and $(M',f')$ are obtained from each other by a \textit{$\GSK$-operation in $X$} if
\begin{enumerate}[(i)]
    \item $M$ has been obtained from $M'$ by a $G$-equivariant cutting and pasting along some admissible submanifold, $\Sigma$, as before in \autoref{def: GSK operation} so $M=M_1 \cup_{\varphi} M_2$ and $M'=M_1 \cup_{\psi} M_2$;
    \item there are $G$-homotopies $f \big|_{M_i} \simeq f' \big|_{M_i}$ for $i=1,2$.
\end{enumerate}
Two $G$-manifolds are called \textit{$\GSK$-equivalent in $X$} if one can be obtained from the other by a finite sequence of $\GSK$-operations in $X$. As before, we denote by $\SKg(X)$ the quotient of the group completion $(\mathcal{M}^{G}(X))^{\grp}$ by the $\GSK$-equivalence relation in $X$. If $X$ is a point, we recover the group $\SKg$.
    \end{definition}
    
    \begin{rem} \label{rem cutting in the same component}
      For $X$ a discrete space, the homotopies must be constant and then the second condition simplifies to $f=f'$. This means we can only cut and paste components that have the same image under the map we are considering.  
    \end{rem}

  \begin{lemma}  
  Let $H$ be a subgroup of $G$. There is a monoid isomorphism
  \[
    \mathcal{M}^{G}(G/H) \cong \mathcal{M}^{H}.
  \]
  \end{lemma}

  \begin{proof}
Let $f \colon M \to G/H$ represent a singular $n$-manifold in $G/H$. Then $f^{-1}(eH)$ is an $n$-dimensional $H$-manifold, which gives a well-defined element in  $\mathcal{M}^{H}$. Conversely, if $N$ is an $H$-manifold, then we map it to the class of the following singular manifold in $G/H$:
\[
G \times_{H} N \to G/H
\]
\[
(g,x) \mapsto gH.
\]
It is straightforward to check that these maps are inverse to each other and preserve disjoint union, which finishes the proof. 
  \end{proof}

Combined with   \autoref{rem cutting in the same component}, the previous lemma implies a parallel statement for the $SK$-groups.

\begin{cor} \label{relative SK identification}
 Let $H$ be a subgroup of $G$. Then $\SKg \left(G/H\right) \cong \SKH.$ 
\end{cor}

\subsection{Mackey functor structure} Next, we are going to upgrade equivariant $SK$-groups to a functor from the Burnside category and check that it is a Mackey functor in the sense of  \autoref{def Mackey in terms of spans}. 

\begin{construction}\label{Mackeyconstruction}
The map on objects
\[
\SKg \colon \mathcal{B}_G \to \Ab
\]
that sends a finite $G$-set $X$ to a group $\SKg(X)$ can be extended to a covariant functor as follows. Given a $G$-equivariant span $X \xleftarrow{\alpha} U \xrightarrow{\beta} Y$ in $\mathcal{B}_G$, we define a map 
  \[
   \mathcal{M}^{G}(X) \to \mathcal{M}^{G}(Y)
  \]
   \[
   (M, f) \mapsto (\widetilde{M}, \beta \circ \widetilde{f}),
   \]
where $(\widetilde{M}, \widetilde{f})$ is defined as the pullback of $(M, f)$ along $\alpha$ endowed with the obvious $G$-action
     \begin{center}
         \begin{tikzcd}
            M \arrow{d}[swap]{f} & \widetilde{M} \arrow[dl, phantom, "\llcorner", very near start]  \arrow[d, dashed]{}{{\widetilde{f}}} \arrow[l, swap, dashed]{}{\widetilde{\alpha}} \\
            X & U \arrow{l}{\alpha} \arrow{r}{\beta} & Y.
         \end{tikzcd}
     \end{center}
This map respects the disjoint union and hence, is a monoid map. Since the $SK$-operation can only be performed on the manifold components mapping to the same point the defined map also respects the $SK$-relation and we obtain the induced map on the $SK$-groups
  \[
    (\alpha, \beta) \colon \SKg(X) \to \SKg(Y).
    \]

    \end{construction}
    
   \begin{prop}\label{Mackeyprop} 
      The map $\SKg \colon \mathcal{B}_G \to \Ab$ of \autoref{Mackeyconstruction} is a $G$-Mackey functor. 
    \end{prop}
\begin{proof}
Functoriality follows from the universal property of the pullback and a diagram chase. For additivity, consider two morphisms $X \xleftarrow{\alpha} U \xrightarrow{\beta} Y$ and $X \xleftarrow{\alpha'} U' \xrightarrow{\beta'} Y$ in $\mathcal{B}_G$ and the corresponding map induced by their sum 
   \begin{center}
         \begin{tikzcd}
            M \arrow{d}[swap]{f} & \widetilde{M} \sqcup \widetilde{M'} \arrow[dl, phantom, "\llcorner", very near start]  \arrow[d, dashed]{}{\widetilde{f} \sqcup \widetilde{f'}} \arrow[l, swap, dashed]{}{\widetilde{\alpha} \sqcup \widetilde{\alpha}'} \\
            X & U \sqcup U' \arrow{l}{\alpha \sqcup \alpha'} \arrow{r}{\beta \sqcup \beta'} & Y.
         \end{tikzcd}
     \end{center}
Since the pullback along the disjoint union of maps is a disjoint union of pullbacks, 
     \[[\widetilde{M},\beta\circ\widetilde{f}]\sqcup[\widetilde{M'},\beta'\circ\widetilde{f'}]=[\widetilde{M}\sqcup\widetilde{M'},(\beta\sqcup\beta')\circ(\widetilde{f}\sqcup\widetilde{f'})]\]
    in $\SKg(Y)$ and hence $(\alpha \sqcup \alpha', \beta \sqcup \beta')$ is a sum of the maps $(\alpha, \beta)$ and $(\alpha', \beta'),$ as desired. 
\end{proof}

The collection of groups $\{ \SKH,~  H \leq G \}$ has geometrically defined restriction, induction and conjugation homomorphisms: restriction of an action to a subgroup, extension of an action to a bigger group by taking the balanced product, and transporting an action to an isomorphic group.  By translating our Mackey functor structure into the axiomatic definition of \cite{webb2000guide}, via the identification $\SKg(G/H)\cong \SKH$, we recover the geometric description of the Mackey functor described above.

\begin{prop}\label{prop:geom}
The restriction, $\res^{H}_{K}$, transfer, $\tr^{H}_{K}$, and conjugation, $c_g$, homomorphisms arising from the Mackey functor structure on the groups $\SKH$ with $H \leq G$ are consistent with the natural restriction, induction, and conjugation homomorphisms described above.
\end{prop}
\begin{proof}   

Let $H$ and $K$ be subgroups of $G$ with $K\leq H$. Let $\pi \colon G/K\to G/H$ be the  equivariant map sending $eK$ to $eH$.
\subsubsection*{Restriction}The restriction map, $\res^H_K$, is obtained by applying the Mackey functor $\SKg$ to the span
\[
G/H \xleftarrow{\pi} G/K \xrightarrow[]{\id} G/K.
\] 

Let $[f \colon M \to G/H]$ be an element in $\SKg(G/H)$. By \autoref{Mackeyconstruction} of the functor $\SKg$ it will be mapped to $[\widetilde{f} \colon \widetilde{M} \to G/K]$ given by the pullback
\[
         \begin{tikzcd}
            M \arrow{d}[swap]{f} & \widetilde{M} \arrow[dl, phantom, "\llcorner", very near start] \arrow[d, dashed]{}{{\widetilde{f}}} \arrow[l, swap, dashed]{}{\widetilde{\pi}} \\
            G/H & G/K \arrow{l}{\pi}.
         \end{tikzcd}
\]
Under the identification of \autoref{relative SK identification}, this corresponds to mapping a class of an $H$-manifold $f^{-1}(eH) \in \SKH$ to the class of a $K$-manifold $\widetilde{f}^{-1}(eK) \in \SKk$. Since the diagram is a pullback,  $\widetilde{f}^{-1}(eK) \cong f^{-1}(eH)$ and the map $\SKH \to \SKk$ is given by the restriction of the $H$-action to a $K$-action.

\subsubsection*{Transfer} To get the transfer, $\tr^{H}_{K}$, we apply our Mackey functor to the span 
\[
G/K \xleftarrow{\id} G/K \xrightarrow{\pi} G/H.
\]
An element $[f \colon M \to G/K] \in \SKg(G/K)$ will be mapped to $[\pi \circ f \colon M \to G/H] \in \SKg(G/H)$. Under the identification of \autoref{relative SK identification}, this corresponds to mapping the class of $N \coloneqq f^{-1}(eK) \in SK^K$  to $f^{-1}(eH) \cong H \times_{K} N \in \SKH.$ Hence, the transfer map on the $SK$-groups is given by the induction map $[N]\mapsto [H\times_{K}N]$.
\subsubsection*{Conjugation.}  To obtain conjugation, $c_g$, we apply our Mackey functor to the span 
\[
G/H \xleftarrow{\id} G/H \xrightarrow{\cong} G/gHg^{-1}.
\] 
Using the identification of \autoref{relative SK identification}, we get a map $\SKH \to SK^{gHg^{-1}}$ given by sending an $H$-manifold to a $gHg^{-1}$-manifold via the identification $H\cong gHg^{-1}$.\end{proof}

\section{Scissors congruence $K$-theory of $G$-manifolds} \label{section :equivariant squares spectrum}

In this section, we use $K$-theory of squares categories to define a squares $K$-theory spectrum for the category of smooth compact $G$-manifolds with boundary, and smooth $G$-maps for a finite group $G$. This recovers the construction of \cite{witpaper} in the case $G=e$. The arguments  from \cite{witpaper} hold almost identically in the presence of equivariance, and we will often refer the reader to that paper for details of proofs.

\subsection{Review of $K$-theory for squares categories} We recall the definition of squares categories and the construction of its $K$-theory spectrum from \cite{CKMZsquarescategories}.

\begin{definition}\label{def:squarescat} A \emph{simple double category} is a small double category whose 2-cells are uniquely determined by their boundaries. Concretely, a simple double category $\calC$ consists of:
	\begin{itemize}
		\item \emph{Horizontal and Vertical Categories.} A pair of categories $(\Hc,\Vc )$ with the same objects as $\calC$. We call $\Hc$ the \emph{horizontal category} and $\Vc$ the \emph{vertical category} of $\calC$.  Morphisms of $\Hc$ are denoted $\rtail$, morphisms of $\Vc$  are denoted $\hookrightarrow$.
		\item \emph{Distinguished Squares.} A collection of square diagrams
		\[\dsquare{A}{B}{C}{D}\]
		that  are required to be closed under horizontal and vertical composition. We also require that for any $f\colon A\rtail B$, and any $g\colon A\htail B$, the squares
		\[	\begin{tikzcd}
		A \ar[r,>->,"f"] \ar[dr,phantom,"\square"] \ar[d,hook,swap,"="] & B \ar[d,hook,"="] \\
		A \ar[r,>->,"f"]& B
		\end{tikzcd} \qquad  \text{and} \qquad	\begin{tikzcd}
		A \ar[r,>->,"="] \ar[dr,phantom,"\square"] \ar[d,hook,swap,"g"]& A \ar[d,hook,"g"] \\
		B \ar[r, >->,"="]& B
		\end{tikzcd}\]
		must be distinguished. 

	\end{itemize}
A functor $F \colon \calC \to \calC'$ of simple double categories is a pair of functors $F_{\Hc} \colon \Hc \to \Hc'$ and $F_{\Vc} \colon \Vc \to \Vc'$ which agree on objects and preserve distinguished squares.
\end{definition}
\begin{definition} A \emph{squares category} is a simple double category $\calC=(\Hc,\Vc)$ with a chosen basepoint $O$ that is initial in both $\Hc$ and $\Vc$. A functor of squares categories is a functor of
simple double categories which preserves the basepoint.
\end{definition}

We will now describe the construction of the $K$-theory spectrum associated to a squares category $\C$. As usual, write $[k]$ for the category $0 \to 1 \to \cdots \to k$.  Let $T_k\C=\mathbf{hFun}([k],\C)$ be the category whose objects are the horizontal functors $[k]\to\Hc$ (that is, horizontal maps $C_0\rightarrowtail C_1\rightarrowtail\cdots\rightarrowtail C_k$) and morphisms the vertical distinguished transformations. A vertical distinguished transformation between horizontal functors $F,G\colon [k]\to\Hc$ is a choice of a vertical morphism $F(i)\hookrightarrow G(i)$ for every $0\leq i\leq k$ such that the square
\[\dsquare{F(i)}{F(i+1)}{G(i)}{G(i+1)}\]
is distinguished for every $0\leq i\leq k-1$. These assemble into a simplicial category, $T_\bullet \C$, and the $K$-theory space is defined as 
 \[K^\square(\C) = \Omega_O |N_\bullet T_\bullet\C|,\]
  where $\Omega_O$ is the based loop space based at the object $O \in N_0
  T_0\C$.  By \cite[Theorem 2.5]{CKMZsquarescategories}, $K^\square(\C)$ is an infinite loop space (we will abuse notation when referring to the associated $\Omega$-spectrum).  We recall the computation of $\pi_0$ of this spectrum, which shows that this construction produces a four term relation for each distinguished square. 
\begin{theorem}[{{\cite[Theorem 3.1]{CKMZsquarescategories}}}]\label{thm:PresK0} Let $\calC$ be a squares category satisfying the condition
	\begin{itemize}
		\item[($\star$)] For all objects $A,B\in\calC$, there exists some object $X$ and distinguished squares 
		\[\dsquare{O}{A}{B}{X}\qquad \text{and} \qquad \dsquare{O}{B}{A}{X}.\]
	\end{itemize}
Then 
$$K_0^\square(\calC)\cong \Z(\ob\calC)/\sim$$
where $\sim$ is the relation $[O]=0$ and for every distinguished square
\[\dsquare{A}{B}{C}{D}\qquad \text{we have}\quad [A]+[D]=[B]+[C].\]	
\end{theorem}

\subsection{The squares category of equivariant manifolds with boundary}
Let $\Mfld$ be the category of smooth compact $G$-manifolds with boundary, and smooth $G$-maps. We will endow this category with a squares category structure. We start with a preliminary definition of the kind of embeddings that we will allow as our horizontal and vertical morphisms. Basically, these are the kind of embeddings that include a piece in a $G$-$SK$ operation.

\begin{definition}
    An \emph{equivariant $SK$-embedding} $N\to M$ between $G$-manifolds is a smooth $G$-embedding $f\colon N\to M$ such that each connected component of $\partial N$ is either mapped entirely onto a boundary component of $M$ or entirely into the interior of $M$. 
\end{definition}

Next, we equip $\Mfld$ with a squares category structure which captures the equivariant cut and paste relation. 

\begin{definition}\label{def:Mfldn} 
Let $\Mfld$ be the category of smooth compact $G$-manifolds with boundary, and smooth $G$-maps. We define a squares category structure as follows:
\begin{itemize}
	\item Define the horizontal and vertical categories $h\Mfld$ and $v\Mfld$ to have the same objects as $\Mfld$ and morphisms the $G$-$SK$-embeddings.
	\item The \emph{distinguished squares} in $\Mfld$ are commutative squares 
	\[\dsquare{N}{M}{M'}{M\cup_N M'}\]
	that are pushout squares in $\Mfld$.
	\item The \emph{basepoint object} is the empty $G$-manifold, denoted $\emptyset$.
		\end{itemize}	 
\end{definition}

\begin{example}\label{squareexamples} We note that in particular we have the following distinguished squares in $\Mfld$: 
	\begin{enumerate}
		\item \emph{Diffeomorphism squares.} Suppose $\phi\colon M\to M'$ is a $G$-diffeomorphism. Then
		\[\dsquaref{\emptyset}{M}{\emptyset}{M'}{}{}{}{\phi}\]
		is a distinguished square.
		\item \emph{Coproduct squares.} Given $G$-manifolds
		$M$ and $M'$, 
		\[	\dsquare{\emptyset}{M}{M'}{M\sqcup M'}\qquad \text{and} \qquad	\dsquare{\emptyset}{M'}{M}{M\sqcup M'}\] 
		are distinguished squares, where $M\sqcup M'$ is the disjoint union of $G$-manifolds. 
	\end{enumerate}	
\end{example}

\begin{prop}\label{claim:Mfldn} $\Mfld$ with the structure defined in \autoref{def:Mfldn} is a squares category satisfying condition ($\star$) of \autoref{thm:PresK0}.
\end{prop}

Checking the axioms is straightforward and analogous to the non-equivariant case \cite[Proposition 4.3]{witpaper}, so we do not repeat it here. Condition $(\star)$ is satisfied by the coproduct squares described in \autoref{squareexamples}.

\begin{theorem}\label{thm:PresSK} There is a group isomorphism
	\[\Kso{\Mfld}\cong \SKG.\]
\end{theorem}

We refer the reader to the proof of \cite[Theorem 4.4]{witpaper}, which holds here using the description from \autoref{thm:PresK0}, with only two small caveats. When showing the relations in $\Kso{\Mfld}$ imply those in $\SKG$, the authors of \cite{witpaper} used squares which glue manifolds together along collared neighborhoods of the boundary. Now we need to use equivariant collars (see e.g. \cite[Theorem 21.2]{conner1979differentiable} or \cite{equivcollar}). When showing the relations in $\SKG$ imply those in $\Kso{\Mfld}$, the authors of \cite{witpaper} used closures of complements of $SK$-embeddings. We note that for equivariant embeddings, complements of the image of the embeddings are $G$-invariant subspaces, and thus, their closures are $G$-invariant submanifolds. Combining these two observations with the proof of \cite[Theorem 4.4]{witpaper} gives the desired result. 

\begin{rem} The standing hypothesis of this paper is that $G$ is a finite group, but in fact \autoref{thm:PresSK} holds for any $G$ with an equivariant collar theorem, for example, compact Lie groups. However, a genuine $G$-spectrum with fixed points given by $K^\square(\Mfld)$ is only expected for finite $G$: lifting the Mackey functor structure on the groups $SK^H$ to the level of spectra requires constructing a spectral Mackey functor based on squares categories. Spectral Mackey functors are known to model a $G$-spectra when $G$ is a finite group, and while spectral Mackey functors based on categorical Mackey functors of symmetric monoidal categories or Waldhausen categories have been used in recent applications \cite{BO, clark, MMEquivariantATheory, MMhcob}, a framework of Mackey functors of squares categories was only very recently worked out in \cite{CalleChan}, where the authors construct the $G$-spectrum whose existence we conjectured for finite groups $G$.
\end{rem}

\subsection{Map to equivariant $A$-theory} 

Let $G$ be a finite group and let $X$ be a $G$-space. Denote by $\mathcal{R}^G(X)$ the Waldhausen category of $G$-retractive spaces over $X$ which are dominated by finite relative $G$-CW complexes with $G$-equivariant maps over and under $X$ as morphisms. Weak equivalences and cofibrations are given by $G$-homotopy equivalences and $G$-cofibrations. The authors of \cite{MMEquivariantATheory} show that the spectra $K(\mathcal{R}^H(X))$, for varying subgroups $H\leq G$, assemble into a spectral Mackey functor, which corresponds to a genuine $G$-spectrum, called the equivariant $A$-theory of $X$. Here $K$ is the $K$-theory associated to the $S_\bullet$-construction of \cite{waldhausen}. 

The fixed points of the equivariant $A$-theory are
$ A_G(X)^H \simeq K(\mathcal{R}^H(X)).$  By \cite{badzioch2017fixed}, the fixed points have a tom Dieck style splitting, and when $X=\ast$ this specializes to
\begin{equation}\label{tomDieck}A_G(*)^H \simeq \displaystyle\prod_{(J) \in C_H} A(BWJ),
\end{equation} 
where $WJ$ is the Weyl group of $J$ in $H$ and the product runs over conjugacy classes of subgroups $J$ of $H$.

Since $\mathcal{R}^H(*)$ is a Waldhausen category, it can be endowed with a squares structure by \cite[Proposition 2.10]{CKMZsquarescategories}, such that
\[
K(\mathcal{R}^H(*)) \simeq K^\square(\mathcal{R}^H(*)^\square).
\]
Explicitly, the squares structure on the category $\mathcal{R}^H(*)^\square$ is given as follows. The basepoint is $*$. The horizontal maps are the cofibrations in $\mathcal{R}^H(*)$, in essence, $H$-equivariant  basepoint preserving maps $X\to Y$ such that $X^J \to Y^J$ is a cofibration for each $J \leq H$. The vertical maps are any morphisms in $\mathcal{R}^H(*)$. The distinguished squares are given by the homotopy pushouts, namely those squares
\[
    \begin{tikzcd}
    A \arrow[r, tail] \arrow[d] & B \arrow[d]\\
    C \arrow[r,tail]& D
    \end{tikzcd}
\]  
for which the map  $(B \cup_A C)^J\to D^J$ is a weak equivalence for all $J\leq H$. The following lemma is immediate since the horizontal morphisms in $\MfldH$ are cofibrations and distinguished squares in $\MfldH$ are pushout squares.

\begin{lem}\label{alpha}
    The map 
\[
\MfldH \rightarrow \mathcal{R}^H(*)^\square
\]
\[
M \mapsto M_+ \coloneqq M \sqcup *,
\]
is a map of squares categories. Hence, it induces a map of spectra
\[
\alpha_H \colon K^{\square}(\MfldH)\rightarrow K^\square(\mathcal{R}^H(*)^\square) \simeq A_G(*)^H.
\]
\end{lem}

We trace this map on $\pi_0$ through the tom Dieck style splitting of the fixed points of the equivariant $A$-theory from \autoref{tomDieck}. For this, we need an explicit description of the map 
  \[
 t_J \colon A_G(*)^H \xrightarrow[]{\simeq} \prod_{(J) \in C_H} A(BWJ) \xrightarrow[]{pr_J} A(BWJ)
  \]
  constructed in \cite{badzioch2017fixed}.

\begin{lemma}\label{tj}
Let $C_H$ be the set of conjugacy classes of subgroups of $H$. For each $(J) \in C_H$, the map induced from the tom Dieck style splitting 
\[
t_J \colon A_G(*)^H\to A(BWJ)
\]
is induced by an exact functor of Waldhausen categories 
\[
\mathcal{R}^H(*) \to \mathcal{R}(BWJ)
\]
 that sends $X$, a retractive $H$-space over a point, to $(X^{J}/\cup_{K \gneq J} X^{K})_{hWJ}$, a retractive space over $BWJ$. 
\end{lemma}

\begin{proof}
 We will not reproduce the full proof from \cite{badzioch2017fixed}, but we will guide the reader on how to reconstruct the desired map. As before, let $\mathcal{R}^H(*)$ be the category of finitely dominated retractive $H$-spaces over a point. Given a subgroup $J \leq H$, let $\mathcal{R}_J^H(*)$ denote the full subcategory of $\mathcal{R}^H(*)$ consisting of those retractive $H$-spaces $Y$ such that the stabilizer of any point in $Y\backslash \{\ast\}$ is a conjugate of $J$.
 The map $t_J$ is obtained by applying $K$-theory to the following composition of exact functors constructed in the proofs of the cited propositions:
  \[
\mathcal{R}^H(*) \xrightarrow[\text{\cite[2.1]{badzioch2017fixed}}]{\textrm{I}} \mathcal{R}_J^H(*)\xrightarrow[\text{\cite[2.2]{badzioch2017fixed}}]{\textrm{II}} \mathcal{R}_{\{e\}}^{WJ}(*) \xrightarrow[\text{\cite[2.3]{badzioch2017fixed}}]{\textrm{III}} \mathcal{R}^{WJ}(EWJ \times *) \xrightarrow[\text{\cite[2.4]{badzioch2017fixed}}]{\textrm{IV}} \mathcal{R}(BWJ).
 \]

\noindent We recall the definitions of these functors. Map I is not constructed explicitly, but proceeds by induction. Order the set of conjugacy classes of subgroups of $H$
   \[
C_H=\{(e)=(H_0), (H_1), \cdots, (H_n)=(H) \}
   \]
in such a way that if $H_i$ is conjugate to a subgroup of $H_j$, which we denote by $(H_i) \leq (H_j)$, then $i \leq j$.
Denote by $\mathcal{R}^H_{\leq i}$ the full subcategory  of $\mathcal{R}^H(\ast)$ consisting of those $H$-spaces $Y$ such that the stabilizer of any point in $Y\backslash \{\ast\}$ is a conjugate of a subgroup of $H_i$. In \cite[Proposition 2.1]{badzioch2017fixed}, the authors prove that there is a weak equivalence 
\[
K(\mathcal{R}^H_{\leq i}) \xrightarrow[]{\simeq} K(\mathcal{R}^H_{\leq i-1}) \times K(\mathcal{R}^H_{H_i}(\ast))  
\]
induced by the following functor:
\[
\mathcal{R}^H_{\leq i} \to \mathcal{R}^H_{\leq i-1} \times \mathcal{R}^H_{H_i}(\ast)
\]
\[
Y \mapsto (Y/Y^{(H_i)}, Y^{(H_i)}). 
\]
The notation $Y^{(H_i)}$ means the union of fixed points of $Y$ over all subgroups of $H$ conjugate to $H_i.$ This way we ensure that it is indeed an $H$-invariant subspace.

Let $X \in \mathcal{R}^H(\ast)=\mathcal{R}^H_{\leq n}$. To determine the image of $X$ in $\mathcal{R}^H_J(\ast)$ for a subgroup $J=H_j$, descend from $\mathcal{R}^H_{\leq n}$ to $\mathcal{R}^H_{\leq j}$ via 
\[
\mathcal{R}^H_{\leq n} \to \mathcal{R}^H_{\leq n-1} \to \mathcal{R}^H_{\leq n-2} \to \ldots \to \mathcal{R}^H_{\leq j}
\]
\[
X \mapsto X/X^{(H_n)} \mapsto X/(X^{(H_n)} \cup X^{(H_{n-1})}) \mapsto \ldots \mapsto X/(X^{(H_n)} \cup X^{(H_{n-1})} \cup \ldots \cup X^{(H_{j+1})}). 
\]
Here we use that $ (Z/A) \Big/  (Z/A)^{(H_l)} =Z /(A \cup Z^{(H_l)})$ for any $H$-invariant subspace $A \subseteq Z$.
Finally, we map to the $\mathcal{R}^H_{J}(\ast)$ component by sending $X/(X^{(H_n)} \cup X^{(H_{n-1})} \cup \ldots \cup X^{(H_{j+1})})$ to its $(J)$-fixed points, which can be rewritten as
\[
X^{(J)}/\cup_{(K) \gneq (J)} X^{(K)}.
\]

\noindent The exact functor $\mathcal{R}_J^H(*) \xrightarrow{\mathrm{II}} \mathcal{R}_{\{e\}}^{WJ}(*)$ sends a retractive space $Y \in \mathcal{R}_J^H(*)$ to its fixed points subspace $Y^J \in \mathcal{R}_{\{e\}}^{WJ}(*)$. 

\noindent The exact functor $\mathcal{R}_{\{e\}}^{WJ}(*) \xrightarrow{\mathrm{III}} \mathcal{R}^{WJ}(EWJ \times *)$ sends a retractive space $Y$ to $EWJ \times Y$. 

\noindent The exact functor 
 $\mathcal{R}^{WJ}(EWJ \times *) \xrightarrow{\mathrm{IV}} \mathcal{R}(EWJ/WJ)$ sends $Y$ to  $Y/WJ$. 
 
 Let $X \in \mathcal{R}^{H}(*)$ be a retractive $H$-space. Note that under the composite of the four maps labeled \textrm{I}, \textrm{II}, \textrm{III}, \textrm{IV}, $X$ gets mapped to $(X^{J}/\cup_{K \gneq J} X^{K})_{hWJ}$, as desired.\end{proof}

In \cite{witpaper}, the map $K^\square(\Mnfld_n^\d)\to A(\ast)$ was followed by the linearization map in order to define a map
$$K^\square(\Mnfld_n^\d)\to K(\Z),$$ which on $\pi_0$ is the map $SK_n^\d\to \Z$, that sends a manifold $M$ to its Euler characteristic.

Recall that the linearization map $A(X) \xrightarrow[]{l} K(\mathbb{Z}[\pi_1(X)])$ is induced by the functor 
\[
\mathcal{R}(X) \to  \ChPerf(\mathbb{Z}[\pi_1(X)])
\]
to the category of perfect chain complexes, that sends $Y$, a retractive space over $X$, to $C_*(\widetilde{Y}, \widetilde{X})$. Here $\widetilde{X}$ denotes the universal cover of $X$ and $\widetilde{Y}$ is the pullback of $\widetilde{X}$ along the retraction $Y \to X$. Note that $\widetilde{Y}$ is a retractive space over $\widetilde{X}$. 

\begin{lemma}\label{linEuler}
On $\pi_0$, the composition of linearization and the map induced by augmentation, 
\[A_0(X) \xrightarrow{l}  K_0(\Z[\pi_1(X)]) \xrightarrow{\aug} K_0(\Z) \cong \Z
\]
sends the class of $Y$ to $\chi(Y,X).$
\end{lemma}
\begin{proof}
  The base change $\ChPerf(\Z[\pi_1(X)]) \to \ChPerf(\Z)$ along the augmentation map $\Z[\pi_1(X)] \to \Z$ corresponds to taking the orbits of the $\pi_1$-action. Therefore,
  \[
  \aug(C_*(\widetilde{X}))=C_*(\widetilde{X})/ \pi_1(X) \simeq C_*(X),
  \]
  \[
  \aug(C_*(\widetilde{Y}))=C_*(\widetilde{Y})/ \pi_1(X)\simeq C_*(Y),
  \]
  \[
  \aug(C_*(\widetilde{X}, \widetilde{Y}))=C_*(\widetilde{X}, \widetilde{Y})/ \pi_1(X)=C_*(\widetilde{X})/ \pi_1(X) \Big/ C_*(\widetilde{Y}) / \pi_1(X) \simeq C_*(X, Y).
  \]
  The identification $K_0(\ChPerf(\Z)) \cong \Z$ sends the class of a chain complex to its Euler characteristic, which finishes the proof.
\end{proof}

The authors of \cite{CCMlinearization} introduce a $K$-theory of coefficient systems of rings, which takes values in genuine $G$-spectra for a finite group $G$. This $K$-theory of the constant coefficient system $\underline{\Z}$ is the target of a genuine equivariant linearization map from equivariant $A$-theory, $$L\colon A_G(*)\to K_G(\underline{\mathbb{Z}}).$$ The fixed points of the $G$-spectrum $K_G(\underline{\mathbb{Z}})$ admit a tom Dieck style splitting compatible with that of $A_G(*)$. Upon passage to fixed points, the equivariant linearization map respects the splittings and recovers the non-equivariant linearization maps. To be precise, the following diagram commutes:
\begin{equation} \label{diag: splitting}
    \begin{tikzcd}
        A_G(*)^H  \arrow[r]{}{L^H} \arrow[d,"\simeq"]&K_G(\underline{\mathbb{Z}})^H\arrow[d,"\simeq"]\\
        \displaystyle\prod_{(J) \in C_H} A(BWJ) \arrow[r]{}{\prod l_J} & \displaystyle\prod_{(J)\in C_H} K(\mathbb{Z}[WJ]).
    \end{tikzcd}
\end{equation} 
The augmentation map $\Z [WJ] \to \Z$ induces the following splitting on $K_0$:
\[
K_0(\mathbb{Z}[WJ]) \cong K_0(\Z) \oplus \widetilde{K_0}(\mathbb{Z}[WJ]) \cong \Z \oplus \widetilde{K_0}(\mathbb{Z}[WJ]).
\]
By putting all of these together we obtain a splitting (see \cite[Remark 4.9]{CCMlinearization})
\[
\displaystyle\prod_{(J)\in C_H} K_0(\mathbb{Z}[WJ]) \cong \Big( \displaystyle\prod_{(J)\in C_H} \Z \Big) \oplus \big( \displaystyle\prod_{(J)\in C_H} \widetilde{K_0}(\mathbb{Z}[WJ]) \Big) \cong \Burn(H) \oplus \big( \displaystyle\prod_{(J)\in C_H} \widetilde{K_0}(\mathbb{Z}[WJ]) \Big),
\]
where $\Burn(H)$ is the Burnside ring.

\begin{theorem}
    There is a map of $K$-theory spectra
    \[
    K^{\square}(\MfldH) \to K_G(\underline{\mathbb{Z}})^H
    \]
    that on $\pi_0$ agrees with the equivariant Euler characteristic on the $\Burn(H)$-component of the target.
\end{theorem}
\begin{proof}
Consider the map $K^{\square}(\MfldH) \to K_G(\underline{\mathbb{Z}})^H$ given by $L^H\circ \alpha _H$, where $\alpha_H$ is the map from \autoref{alpha}.
We obtain an explicit description for the following composite map
    \begin{equation} \label{eq: long composition map2}
K^{\square}(\MfldH) \xrightarrow[]{\alpha_H} A_G(*)^H \xrightarrow[]{\simeq} \prod_{(J) \in C_H} A(BWJ) \xrightarrow[]{pr_J} A(BWJ) \xrightarrow[]{l_J} K(\Z [WJ])
    \end{equation}
at the level of $\pi_0$ for each $(J) \in C_H$. Recall that 
$K_0^{\square}(\MfldH)\cong \SKH$, and on $\pi_0$ the map 
\[
\alpha_H \colon K^{\square}(\MfldH) \rightarrow A_G(*)^H
\]
sends the class of an $H$-manifold $M$ to the class of the $H$-space $M_+$. 
By \autoref{tj}, the map
\[
t_J \colon A_G(*)^H\to A(BWJ) 
\] 
sends the class of $M_+$ to  $((M_+)^{J}/\cup_{K \gneq J} (M_+)^{K})_{hWJ}$, which is a retractive space over $BWJ$. Since $(M_+)^{J}/\cup_{K \gneq J} (M_+)^{K}$ has free $WJ$-action away from the point where the subspace is collapsed, 
\begin{equation}
    \big( (M_+^{J}/\cup_{K \gneq J} M_+^{K}) \times_{WJ} EWJ \big)/BWJ  \simeq (M_+^{J}/\cup_{K \gneq J} M_+^{K})/WJ.
\end{equation}
\noindent Hence, the following relative Euler characteristics agree
\begin{eqnarray*}
    \chi((M_+^{J}/\cup_{K \gneq J} M_+^{K}) \times_{WJ} EWJ, BWJ )&=&\chi(M_+^{J}/WJ, \bigcup_{K \gneq J} M_+^{K}/WJ)\\
    &=&\chi(M^{J}/WJ, \bigcup_{K \gneq J} M^{K}/WJ).
\end{eqnarray*}
Thus, by \autoref{linEuler}, under the linearization map followed by the augmentation map,
\[
A_0(BWJ) \xrightarrow[]{l_J} K_0(\mathbb{Z}[WJ]) \cong \Z \oplus \widetilde{K_0}(\mathbb{Z}[WJ]) \xrightarrow[]{\aug_J}  \Z
\]
we have the following assignment
$$(M_+^{J}/\bigcup_{K \gneq J} M_+^{K})_{hWJ} \ \mapsto\ 
 \chi(M^{J}/WJ,\  \bigcup_{K \gneq J} M^{K}/WJ).$$

Therefore, using the description from \autoref{Eulerformula}, on the level of $\pi_0$, the composite map 
\[
K^{\square}(\MfldH) \to K_G(\underline{\mathbb{Z}})^H \cong \displaystyle\prod_{(J)\in C_H} K(\mathbb{Z}[WJ]) \xrightarrow[]{\prod \aug_J} \Burn(H)
\]
sends the class of $[M]$ to the equivariant Euler characteristic of $M$ in the Burnside ring $\Burn(H)$. 
\end{proof}

 \bibliographystyle{alpha}
  \bibliography{biblio}

\newcommand{\etalchar}[1]{$^{#1}$}
\begin{thebibliography}{HMM{\etalchar{+}}22}

\bibitem[Bar17]{clark}
Clark Barwick.
\newblock Spectral {M}ackey functors and equivariant algebraic {$K$}-theory
  ({I}).
\newblock {\em Adv. Math.}, 304:646--727, 2017.

\bibitem[BD17]{badzioch2017fixed}
Bernard Badzioch and Wojciech Dorabia{\l}a.
\newblock Fixed points of the equivariant algebraic {K}-theory of spaces.
\newblock {\em Proceedings of the American Mathematical Society},
  145(9):3709--3716, 2017.

\bibitem[BGM{\etalchar{+}}24]{bgmmz}
Anna~Marie Bohmann, Teena Gerhardt, Cary Malkiewich, Mona Merling, and Inna
  Zakharevich.
\newblock {A Trace Map on Higher Scissors Congruence Groups}.
\newblock {\em International Mathematics Research Notices}, 08 2024.

\bibitem[BO15]{BO}
Anna~Marie Bohmann and Ang\'elica Osorno.
\newblock Constructing equivariant spectra via categorical {M}ackey functors.
\newblock {\em Algebr. Geom. Topol.}, 15(1):537--563, 2015.

\bibitem[Bre72]{bredon1972introduction}
Glen~Eugene Bredon.
\newblock {\em Introduction to Compact Transformation Groups}.
\newblock Pure and Applied Mathematics. Academic Press, 1972.

\bibitem[CC]{CalleChan}
Maxine Calle and David Chan.
\newblock A genuine {$G$}-spectrum for the cut-and-paste {$K$}-theory of
  {$G$}-manifolds.
\newblock {\em arXiv:2508.03621}.

\bibitem[CCM23]{CCMlinearization}
Maxine Calle, David Chan, and Andres Mejia.
\newblock A linearization map for genuine equivariant algebraic {$K$}-theory.
\newblock {\em arXiv preprint arXiv:2309.08025}, 2023.

\bibitem[CF79]{conner1979differentiable}
Pierre~Euclide Conner and Edwin~Earl Floyd.
\newblock {\em Differentiable periodic maps}, volume~29.
\newblock Springer, 1979.

\bibitem[CKMZ23]{CKMZsquarescategories}
Jonathan Campbell, Josefien Kuijper, Mona Merling, and Inna Zakharevich.
\newblock Algebraic {K}-theory for squares categories.
\newblock {\em arXiv preprint arXiv:2310.02852}, 2023.

\bibitem[Deh01]{dehn}
Max~Wilhelm Dehn.
\newblock Ueber den {R}auminhalt.
\newblock {\em Math. Ann.}, 55(3):465--478, 1901.

\bibitem[Dug19]{dugger2019involutions}
Daniel Dugger.
\newblock Involutions on surfaces.
\newblock {\em Journal of Homotopy and Related Structures}, 14:919--992, 2019.

\bibitem[HK97]{hara1997cutting}
Tamio Hara and Hiroaki Koshikawa.
\newblock Cutting and pasting of {$G$} manifolds with boundary.
\newblock {\em Kyushu Journal of Mathematics}, 51(1):165--178, 1997.

\bibitem[HMM{\etalchar{+}}22]{witpaper}
Renee~S. Hoekzema, Mona Merling, Laura Murray, Carmen Rovi, and Julia Semikina.
\newblock {C}ut and paste invariants of manifolds via algebraic {K}-theory.
\newblock {\em Topology and its Applications}, 316:108105, 2022.
\newblock Women in Topology III.

\bibitem[Ill83]{illman}
S\"{o}ren Illman.
\newblock The equivariant triangulation theorem for actions of compact {L}ie
  groups.
\newblock {\em Mathematische Annalen}, 262:487--501, 1983.

\bibitem[J{\"a}n66]{janich1966differenzierbare}
Klaus J{\"a}nich.
\newblock Differenzierbare {M}annigfaltigkeiten mit {R}and als {O}rbitr{\"a}ume
  differenzierbarer {G}-{M}annigfaltigkeiten ohne {R}and.
\newblock {\em Topology}, 5(4):301--320, 1966.

\bibitem[J{\"a}n69]{janich1969invariants}
Klaus J{\"a}nich.
\newblock On invariants with the {N}ovikov additive property.
\newblock {\em Mathematische Annalen}, 184(1):65--77, 1969.

\bibitem[Kan07]{equivcollar}
Marja Kankaanrinta.
\newblock Equivariant collaring, tubular neighbourhood, and gluing theorems for
  proper {L}ie group action.
\newblock {\em Algebraic and {G}eometric {T}opology}, 7, 2007.

\bibitem[KKNO73]{SKbook}
Ulrich Karras, Matthias Kreck, Walter~David Neumann, and Erich Ossa.
\newblock {\em Cutting and pasting of manifolds; {${\rm SK}$}-groups}.
\newblock Publish or Perish, Inc., Boston, Mass., 1973.
\newblock Mathematics Lecture Series, No. 1.

\bibitem[KLM{\etalchar{+}}24]{klmms}
Alexander Kupers, Ezekiel Lemann, Cary Malkiewich, Jeremy Miller, and Robin~J.
  Sroka.
\newblock Scissors automorphism groups and their homology.
\newblock {\em arXiv preprint arXiv:2408.08081}, 2024.

\bibitem[Kom03]{komiya2003cutting}
Katsuhiro Komiya.
\newblock Cutting and pasting of manifolds into {G}-manifolds.
\newblock {\em Kodai Mathematical Journal}, 26(2):230--243, 2003.

\bibitem[Kos78]{kosniowski}
Czes Kosniowski.
\newblock {\em Actions of finite abelian groups}, volume~18 of {\em Research
  Notes in Mathematics}.
\newblock Pitman (Advanced Publishing Program), Boston, Mass.-London, 1978.

\bibitem[Lin76]{lindner1976remark}
Harald Lindner.
\newblock A remark on {M}ackey-functors.
\newblock {\em manuscripta mathematica}, 18(3):273--278, 1976.

\bibitem[LR03]{LuckRosenberg}
Wolfgang L\"uck and Jonathan Rosenberg.
\newblock Equivariant {E}uler characteristics and {$K$}-homology {E}uler
  classes for proper cocompact {$G$}-manifolds.
\newblock {\em Geom. Topol.}, 7:569--613, 2003.

\bibitem[L{\"u}c05]{luck2005burnside}
Wolfgang L{\"u}ck.
\newblock The {B}urnside ring and equivariant stable cohomotopy for infinite
  groups.
\newblock {\em arXiv preprint math/0504051}, 2005.

\bibitem[Mal22]{scissors_thom}
Cary Malkiewich.
\newblock Higher scissors congruence.
\newblock {\em arXiv preprint arXiv:2210.08082}, 2022.

\bibitem[MM19]{MMEquivariantATheory}
Cary Malkiewich and Mona Merling.
\newblock Equivariant {A}-theory.
\newblock {\em Documenta Mathematica}, 24:815--855, 2019.

\bibitem[MM22]{MMhcob}
Cary Malkiewich and Mona Merling.
\newblock The equivariant parametrized {$h$}-cobordism theorem, the
  non-manifold part.
\newblock {\em Adv. Math.}, 399:Paper No. 108242, 42, 2022.

\bibitem[Row70]{rowlett1971additive}
Russell~Johnston Rowlett.
\newblock {\em Additive invariants of manifolds with boundary}.
\newblock PhD thesis, University of Virginia, 1970.

\bibitem[Syd65]{Sydler}
Jean-Pierre Sydler.
\newblock Conditions n\'{e}cessaires et suffisantes pour l'\'{e}quivalence des
  poly\`edres de l'espace euclidien \`a trois dimensions.
\newblock {\em Comment. Math. Helv.}, 40:43--80, 1965.

\bibitem[tD87]{tomDieckbook}
Tammo tom Dieck.
\newblock {\em Transformation groups}, volume~8 of {\em De Gruyter Studies in
  Mathematics}.
\newblock Walter de Gruyter \& Co., Berlin, 1987.

\bibitem[Wal87]{waldhausen}
Friedhelm Waldhausen.
\newblock Algebraic {$K$}-theory of spaces, concordance, and stable homotopy
  theory.
\newblock In {\em Algebraic topology and algebraic {$K$}-theory ({P}rinceton,
  {N}.{J}., 1983)}, volume 113 of {\em Ann. of Math. Stud.}, pages 392--417.
  Princeton Univ. Press, Princeton, NJ, 1987.

\bibitem[Web00]{webb2000guide}
Peter Webb.
\newblock A guide to {M}ackey functors.
\newblock In {\em Handbook of algebra}, volume~2, pages 805--836. Elsevier,
  2000.

\bibitem[Zak12]{inna_scissorsKth}
Inna Zakharevich.
\newblock {Scissors congruence as $K$-theory}.
\newblock {\em Homology, Homotopy and Applications}, 14(1):181 -- 202, 2012.

\end{thebibliography}

\begingroup%
\setlength{\parskip}{\storeparskip}

\end{document}